\documentclass[11pt,a4paper]{article}

\usepackage[T1]{fontenc}
\usepackage[utf8]{inputenc}
\usepackage{lmodern}
\usepackage{microtype}

\usepackage[margin=2.5cm,
            headheight=14pt,
            headsep=0.8cm,
            footskip=1.0cm]{geometry}

\usepackage{amsmath,amssymb,amsthm,mathtools}
\usepackage{xcolor}
\usepackage{cite}
\usepackage[hidelinks]{hyperref}
\hypersetup{
  pdftitle={Maximal Gaps for Dilated Lacunary Integer Sequences},
  pdfauthor={Yuval Peres and Bohan Yang},
  pdfsubject={Almost-sure maximal gaps for lacunary dilations}
}
\usepackage{aliascnt}

\usepackage{enumitem}
\usepackage{titling}                     

\numberwithin{equation}{section}
\allowdisplaybreaks                      

\theoremstyle{plain}                     

\newtheorem{theorem}{Theorem}[section]

\newaliascnt{proposition}{theorem}
\newtheorem{proposition}[proposition]{Proposition}
\aliascntresetthe{proposition}

\newaliascnt{lemma}{theorem}
\newtheorem{lemma}[lemma]{Lemma}
\aliascntresetthe{lemma}

\newaliascnt{corollary}{theorem}

\aliascntresetthe{corollary}

\theoremstyle{definition}
\newaliascnt{definition}{theorem}
\newtheorem{definition}[definition]{Definition}
\aliascntresetthe{definition}

\theoremstyle{remark}
\newaliascnt{remark}{theorem}
\newtheorem{remark}[remark]{Remark}
\aliascntresetthe{remark}

\newaliascnt{claim}{theorem}

\aliascntresetthe{claim}

\usepackage[nameinlink,capitalize]{cleveref}

\crefname{theorem}{Theorem}{Theorems}
\crefname{proposition}{Proposition}{Propositions}
\crefname{lemma}{Lemma}{Lemmas}
\crefname{corollary}{Corollary}{Corollaries}
\crefname{definition}{Definition}{Definitions}
\crefname{remark}{Remark}{Remarks}
\crefname{claim}{Claim}{Claims}

\newcommand{\bbT}{\mathbb{T}}
\newcommand{\bbR}{\mathbb{R}}
\newcommand{\bbZ}{\mathbb{Z}}
\newcommand{\bbN}{\mathbb{N}}
\newcommand{\Leb}{\lambda}
\newcommand{\ind}{\mathbf{1}}
\newcommand{\E}{\mathbb{E}}
\newcommand{\PP}{\mathbb{P}}

\DeclareMathOperator{\Var}{Var}

\title{Maximal Gaps for Dilated Lacunary Integer Sequences}

\author{
    Yuval Peres\\[2pt]
    \small Beijing Institute of Mathematical Sciences and Applications (BIMSA)
    \and
    Bohan Yang\\[2pt]
    \small Shanghai Institute for Mathematics and Interdisciplinary Sciences (SIMIS)}

\date{\today}

\begin{document}

\maketitle

\begin{abstract}
Let \((a_n)_{n\ge1}\subset\bbN\) be a lacunary sequence,
\(a_{n+1}\ge q a_n\) for \(q>1\).  For \(x\in\bbT\), we study the
maximal empty circular gap \(G_N(x)\) of the finite orbit
\(\{a_1x,\ldots,a_Nx\}\). 
We prove that, for Lebesgue-almost every \(x\),
\[
  \frac12
  \le \liminf_{N\to\infty}\frac{NG_N(x)}{\log N}
  \le \limsup_{N\to\infty}\frac{NG_N(x)}{\log N}
  \le  \frac{q+1}{q-1}\,.
\]
If, in addition, \(a_n\mid a_{n+1}\) for every \(n\), then this can be improved to 
\[
  \lim_{N\to\infty}\frac{NG_N(x)}{\log N}=1
\]
for Lebesgue-almost every \(x\).
\end{abstract}

\section{Introduction}\label{sec:intro}

Let \(\bbT=\bbR/\bbZ\).  Given an increasing sequence
\((a_n)_{n\ge1}\) and a point \(x\in\bbT\), consider the finite set
\begin{equation}\label{eq:orbit}
    A_N(x)=\{a_1x,\ldots,a_Nx\}\subset\bbT .
\end{equation}
The central object of this paper is the maximal circular gap of this point
set,
\begin{equation}\label{eq:gap-def}
    G_N(x)
    =
    \sup\bigl\{|J|:\ J\subset\bbT\text{ is an interval and }
    J\cap A_N(x)=\varnothing\bigr\}.
\end{equation}
Equivalently, \(\frac{1}{2}G_N(x)\) measures the maximal error in approximating targets
  \(t\in\bbT\) by $A_N(x)$, so
\begin{equation}\label{eq:gap-dioph}
    G_N(x)
    =
    2\sup_{t\in\bbT}
    \min_{1\le n\le N}
    \|a_nx-t\|.
\end{equation}


For \(N\) independent uniform points on \(\bbT\), the maximal gap is asymptotic to \((\log N)/N\), see \cite{Devroye}.  Since lacunary dilates $(a_n x)_{n \ge 1}$ often exhibit  properties similar to independent random variables
 for Lebesgue-typical \(x\), it is natural to ask whether $G_N(x)$ has the same asymptotics. 

We now assume that
\((a_n)\subset\bbN\) satisfies the Hadamard lacunarity condition
\begin{equation}\label{eq:H}
    a_{n+1}\ge q a_n,\qquad q>1.
\end{equation}
The maximal-gap problem for dilated lacunary sequences was studied by
Chow--Technau~\cite{ChowTechnau} and Stefanescu~\cite{Stefanescu,StefanescuConvex}.
 In addition to the
Lebesgue case, the corresponding almost-sure statements were also obtained
for probability measures \(\mu\) satisfying suitable Fourier-decay assumptions.
Chow--Technau~\cite{ChowTechnau} proved that, for every \(\varepsilon>0\) and for
\(\mu\)-almost every dilation parameter \(x\),
\[
    G_N(x)
    =
    O\left(\frac{(\log N)^{3+\varepsilon}}{N}\right).
\]
Stefanescu~\cite{Stefanescu} improved this to
\[
    G_N(x)
    =
    O\left(\frac{(\log N)^{2+\varepsilon}}{N}\right),
\]
and later in~\cite{StefanescuConvex} removed the \(\varepsilon\) from the logarithmic exponent in a
multidimensional convex-body version of the problem.  In particular, in
dimension one this gives the bound
\[
    G_N(x)
    =
    O\left(\frac{(\log N)^2}{N}\right)
\]
for \(\mu\)-almost every \(x\).

Here we prove Lebesgue-almost sure upper and lower bounds of order \((\log N)/N\).
Define
\begin{equation}\label{eq:Gamma-main}
    \Gamma :=
    \limsup_{L\to\infty}\sup_{p\ge1} \Gamma(p,L), \, \text{ where } \,
    \Gamma(p,L):=\frac1L\sum_{m=p}^{p+L-1} \; \sum_{n=m+1}^{p+L}\frac{a_m}{a_n} \, .
\end{equation}
Lacunarity gives \(0\le\Gamma\le(q-1)^{-1}\). We discuss $\Gamma$ in \cref{sec:upper}.

\begin{theorem} \label{thm:one-d}
    Assume~\eqref{eq:H}. Then, for Lebesgue-almost every \(x\in\bbT\),
    \begin{equation}\label{eq:main-liminf}
        \liminf_{N\to\infty}\frac{NG_N(x)}{\log N}
        \ge \frac12,
    \end{equation}
    and
    \begin{equation}\label{eq:main-limsup}
        \limsup_{N\to\infty}\frac{NG_N(x)}{\log N}
        \le 1+2\Gamma
        \le \frac{q+1}{q-1}.
    \end{equation}
\end{theorem}
For sequences $a_n=a^n$ with integer $a>1$, we can prove that the limiting constant is the same as in the independent case.  Indeed, divisibility of each $a_{n+1}$ by $a_n$ suffices. 
\begin{theorem}[Divisibility chains]\label{thm:divisible}
    Assume~\eqref{eq:H} and, in addition, suppose that
    \(a_n\mid a_{n+1}\) for every \(n\ge1\).
    Then, for Lebesgue-almost every \(x\in\bbT\),
    \begin{equation}\label{eq:divisible-limit}
        \lim_{N\to\infty}\frac{NG_N(x)}{\log N}=1.
    \end{equation}
\end{theorem}

\begin{remark}
The lower constant \(1/2\) in \cref{thm:one-d} comes from a second-moment
estimate for terminal empty cells.  The upper bound follows from
Paley--Zygmund applied block by block, with the local correlation constant
\(\Gamma\) recording the actual ratio structure.  Under the divisibility
hypothesis, short multiplicative returns can be isolated and the mixed-radix
digit structure yields the sharp constant~\(1\).
\end{remark}


\subsection{Related work and comparison}

The identity \eqref{eq:gap-dioph} gives a direct Diophantine interpretation of
the problem.  For fixed \(x\), the finite orbit
\(\{a_nx:1\le n\le N\}\) is used to approximate arbitrary inhomogeneous
targets \(t\in\bbT\).  Thus \(G_N(x)\) is a finite-time uniform
inhomogeneous approximation radius.

This formulation should be distinguished from the classical shrinking-target
problem studied, for example, by
Pollington--Velani--Zafeiropoulos--Zorin~\cite{PVZZ}.  In that setting one
fixes a target \(t\) and studies whether
\[
|a_nx-t|<\psi(n)
\]
holds infinitely often.  In the present problem the target is not fixed in
advance: the supremum over \(t\) in \eqref{eq:gap-dioph} asks for a uniform
approximation statement over all targets simultaneously.  Related uniform
target-set questions for typical orbits of expanding Markov maps, including
Hausdorff-dimension formulae for uniformly approximable target sets, were
studied by He and Liao~\cite{HeLiao}.  This uniformity is precisely what turns
the present question into a covering-radius, or maximal-gap, problem.

There is also a closely related circle-covering viewpoint.  Recent work of
Hauke--Shubin--Stefanescu--Zafeiropoulos~\cite{HSSZ} studies coverings of the
circle by shrinking intervals centered at points \(\{a_nx\}\), with
applications to metric Diophantine approximation and Littlewood-type problems.
In that setting one studies whether shrinking intervals cover the circle
infinitely often.  By contrast, the present paper studies a finite-time covering
radius: for each \(N\), how large must a common radius be so that the first
\(N\) centers \(a_1x,\ldots,a_Nx\) cover every target point?  Equivalently, we
study the largest hole left by the first \(N\) approximants.  Thus the scale
\[
\frac{\log N}{N}
\]
appears here not as a shrinking-target summability threshold, but as the
finite-time extremal scale for uniform approximation of all targets.

Lacunary sequences have long served as a bridge between deterministic harmonic
analysis and probabilistic behaviour.  We refer to Aistleitner, Berkes and
Tichy~\cite{ABT} for a broad survey, and to Katznelson~\cite{Katznelson} for
Fourier-analytic background.  Alon and Peres~\cite{AlonPeres} developed a
harmonic-analysis method for uniform dilations.  Konyagin--Ruzsa--Schlag
~\cite{KRS} studied uniform distribution properties of dilates of finite
integer sequences, a closely related finite-set dilation problem on the torus.
Peres and Schlag~\cite{PeresSchlag} studied lacunary sequences in connection
with two problems of Erd\H{o}s; their methods include scale separation and
rounding techniques useful in several lacunary problems.  Moshchevitin
~\cite{Moshchevitin} adapted related ideas to sublacunary sequences.

Another related, but distinct, line of work concerns Poissonian local
statistics for lacunary dilates.  Rudnick--Zaharescu~\cite{RZ} studied the
distribution of spacings between fractional parts of lacunary sequences and
proved Poissonian local correlations.  Such results describe local statistics
at microscopic scales, whereas the present paper studies the extremal
finite-time covering radius.

The works closest to the present paper are those of
Chow--Technau~\cite{ChowTechnau} and
Stefanescu~\cite{Stefanescu,StefanescuConvex}, which obtain upper bounds for
maximal gaps of dilated lacunary sequences.  Their conclusions hold for almost
every dilation parameter with respect to probability measures satisfying
suitable Fourier-decay assumptions.  In a complementary deterministic
direction, Stefanescu~\cite{Stefanescu} proved that every lacunary sequence
admits a dilation factor \(x\) for which
\[
G_N(x)=O_q\left(\frac{\log N}{N}\right).
\]
By contrast, the present paper proves both upper and lower bounds of order
\((\log N)/N\) for Lebesgue-almost every dilation factor \(x\).

\subsection{Overview of the method and organization}
\label{subsec:overview}

\Cref{sec:elementary} collects the elementary one-dimensional estimates used
throughout the paper.  The one-point estimate controls the discrepancy of
\(\ind_A(mx)\) on finite unions of intervals, while the local two-point
estimate bounds
\[
    \int_F \ind_J(ax)\ind_J(bx)\,dx
\]
in terms of the ratio \(a/b\) and the boundary complexity of \(F\).  These
estimates are the basic local mixing inputs for both the moving-target lower
bound and the block second-moment upper bound.

The lower bound, proved in \cref{sec:lower}, uses moving targets.  Instead of
fixing a gap and varying \(x\), we introduce a shift parameter and estimate the
\((x,t)\)-probability that all \(N\) points miss \(t+B\).  The main input is
\cref{lem:escape}, which gives
\[
    (1-\Leb_1(B))^N
    +
    O_{q,C}\left(\frac{(\log N)^3}{N}\right)
\]
for \(\Leb_1(B)\le C(\log N)/N\).  The proof uses the elementary mixing estimates
of \cref{sec:elementary} and exact Haar preservation on \(\bbT^2\).  This
product-space step is inspired by the scale-separation method of
Peres--Schlag~\cite{PeresSchlag}, but the moving target gives an explicit
avoidance recursion.

The upper bound, proved in \cref{sec:upper}, partitions the indices into main
blocks and buffers.  On each survivor set, Paley--Zygmund and the two-point
estimate remove a fixed proportion of points.  The off-diagonal contribution
is controlled by \(\Gamma\), which records the actual local ratio structure.
Iterating over \(\asymp\log N\) blocks gives \eqref{eq:main-limsup}.

Finally, \cref{sec:divisible} proves \cref{thm:divisible}.  The sharp upper
bound separates intervals with short multiplicative returns from regular
intervals.  For the lower bound, the divisibility chain supplies independent
mixed-radix digits; a sparse family of intervals with no short returns and a
local-dependence estimate then give an empty interval at every scale
\(\tau(\log N)/N\) with \(\tau<1\).
\section{Elementary one-dimensional estimates}\label{sec:elementary}

\begin{lemma}[One-point mixing]\label{lem:onepoint}
    Let \(F\subset\bbT\) be a finite union of intervals, and let
    \(A\subset\bbT\) be measurable. Then, for every integer \(m\ge1\),
    \begin{equation}\label{eq:onepoint}
        \biggl|\int_F \ind_A(mx)\,dx-\Leb_1(F)\,\Leb_1(A)\biggr|
        \le \frac{\#\partial F}{m}.
    \end{equation}
\end{lemma}

\begin{proof}
    Partition \(\bbT\) into \(I_r=[\frac r m,\frac{r+1}{m})\), \(0\le r<m\).
    On each \(I_r\), the map \(x\mapsto mx\pmod1\) maps \(I_r\) linearly onto
    \(\bbT\).  If \(I_r\) is contained in \(F\) or in \(F^c\), the contribution is
    exact.  At most \(\#\partial F\) cells meet \(\partial F\), and on each such cell the
    absolute discrepancy is at most \(1/m\).
\end{proof}

\begin{lemma}[Local two-point estimate]\label{lem:twopoint}
    Let \(F\subset\bbT\) be a finite union of intervals.  Let
    \(J\subset\bbT\) be an interval of length \(s\le1/2\).  For integers
    \(1\le a<b\), one has
    \begin{equation}\label{eq:twopoint}
        \int_F \ind_J(ax)\,\ind_J(bx)\,dx
        \le
        \Leb_1(F)\Bigl(s^2+s\frac ab\Bigr)
        + \#\partial F\,\frac{s}{a}.
    \end{equation}
\end{lemma}

\begin{proof}
    Write \(J=\xi+[0,s)\pmod1\), with \(\xi\in[0,1)\), and partition
    \(\bbT\) into the shifted \(a\)-adic intervals
    \[
        I_r^\xi
        =
        \left[\frac{r+\xi}{a},\frac{r+\xi+1}{a}\right)\pmod1,
        \qquad 0\le r<a.
    \]
    We first estimate the contribution of those cells \(I_r^\xi\) which are
    contained in \(F\).  On such a cell write
    \[
        x=\frac{r+\xi+y}{a}\pmod1,
        \qquad 0\le y<1.
    \]
    Then
    \[
        ax\equiv \xi+y\pmod1,
        \qquad
        bx-\xi\equiv Ry+\theta_r\pmod1,
        \qquad R=\frac ba,
    \]
    for a phase \(\theta_r\in\bbT\).  Consequently, the two conditions
    \(ax\in J\) and \(bx\in J\) are equivalent to
    \[
        y\in[0,s),
        \qquad
        Ry+\theta_r\in[0,s)\pmod1.
    \]
    Thus the contribution from a full cell \(I_r^\xi\subset F\) is
    \[
        \frac1a\,
        \bigl|\{0\le y<s:Ry+\theta_r\in[0,s)\pmod1\}\bigr|.
    \]

    As \(y\) ranges over \([0,s)\), the variable \(z=Ry+\theta_r\) ranges over
    an interval \(K\subset\bbR\) of length \(Rs\).  Since the periodic
    set \([0,s)+\bbZ\) has measure \(s\) in each unit interval, every
    interval \(K\) of length \(L\) satisfies
    \[
        \Leb_1\bigl(K\cap([0,s)+\bbZ)\bigr)\le Ls+s.
    \]
    Taking \(L=Rs\) and changing variables back from \(z\) to \(y\) gives
    \[
        \Leb_1\bigl(\{0\le y<s:Ry+\theta_r\in[0,s)\pmod1\}\bigr)
        \le
        \frac1R(Rs^2+s)
        =s^2+s\frac ab.
    \]
    Hence the full cells contribute at most
    \[
        \Leb_1(F)\left(s^2+s\frac ab\right).
    \]

    At most \(\#\partial F\) shifted cells meet \(\partial F\).  On each such
    cell, the condition \(ax\in J\) restricts \(x\) to a set of measure at
    most \(s/a\).  Their total contribution is therefore at most
    \(\#\partial F\,s/a\), which proves~\eqref{eq:twopoint}.
\end{proof}
\section{Moving targets and the lower bound}\label{sec:lower}

Let $B\subset\bbT$ be a union of at most two intervals, and write
$b=\Leb_1(B)$.  For $n\ge1$, define
\begin{equation}\label{eq:En-def}
    E_n(B)=\{(x,t)\in\bbT^2: a_nx-t\in B\},
\end{equation}
and
\begin{equation}\label{eq:PN-def}
    P_N(B)=\Leb_2\!\left(\bigcap_{n=1}^N E_n(B)^c\right).
\end{equation}

\begin{remark}\label{rem:PS}
    The moving-target formalism developed below is inspired by the ``scale
    separation'' technique of Peres and Schlag~\cite[Section~3]{PeresSchlag},
    who studied Diophantine approximation properties of lacunary sequences.
    In that work, for a fixed target $B$, the event $\{a_nx\in B\}$ is
    analysed by splitting the index set into blocks separated by buffers of
    length~$h$, and controlling the boundary complexity of past constraints
    via the geometric decay $a_{n-h}\le q^{-h}a_n$.  Our contribution in
    \cref{lem:escape} is twofold: we formulate the problem on the
    product space $\bbT^2$ with a \emph{moving} target $t+B$, which converts
    the one-dimensional avoidance problem into a joint $(x,t)$-avoidance
    problem with exact Haar measure preservation; and we obtain a sharp
    recursion $|P_k-(1-b)P_{k-1}|\le 2hb^2+O_q(q^{-h})$ that iterates to the
    clean estimate~\eqref{eq:escape}.  The integer determinant argument
    ($a_k-a_j\neq0$) guaranteeing $\Leb_2(E_j\cap E_k)=b^2$ is the same
    algebraic fact used by Peres and Schlag to control intersections of
    dilated level sets.
\end{remark}

\begin{lemma}[Moving-target escape]\label{lem:escape}
    Let
    \[
        h=\bigl\lceil \tfrac{10\log N}{\log q}\bigr\rceil,
    \]
    and assume $h<N$. If $B$ is a union of at most two intervals and
    $b=\Leb_1(B)$, then
    \begin{equation}\label{eq:escape}
        \bigl|P_N(B)-(1-b)^N\bigr|
        \le
        hb+2Nhb^2+\frac{4q}{q-1}\,Nq^{-h}.
    \end{equation}
    In particular, if $b\le C_0(\log N)/N$, then
    \begin{equation}\label{eq:escape-special}
        P_N(B)=(1-b)^N+O_{q,C_0}\!\left(\frac{(\log N)^3}{N}\right).
    \end{equation}
\end{lemma}

\begin{proof}
    Work on the probability space $(\bbT^2,\mathcal B,\PP)$ where
    $\PP=\Leb_2$ is the normalised Haar measure.
    For $n\ge1$ set
    \[
        E_n=\{(x,t)\in\bbT^2: a_nx-t\in B\},\qquad
        \PP(E_n)=\Leb_1(B)=b .
    \]
    Define the \emph{survivor set} after $k$ steps and its measure
    \[
        S_k=\bigcap_{n=1}^k E_n^c,\qquad
        P_k=\PP(S_k).
    \]

    Fix $k>h$.  Split the index set $\{1,\dots,k-1\}$ into the
    \emph{remote past} (indices at least $h$ steps away) and the
    \emph{near past} (the buffer of length $h-1$):
    \[
        R_k=\bigcap_{j\le k-h}E_j^c,\qquad
        L_k=\bigcap_{k-h<j<k}E_j^c .
    \]
    We have the exact identities
    \begin{equation}\label{eq:id1}
        S_{k-1}=R_k\cap L_k,\qquad
        R_k=(R_k\cap L_k)\;\cup\;(R_k\cap L_k^c),
    \end{equation}
    and $L_k^c=\bigcup_{k-h<j<k}E_j$.
The proof is based on a one-step comparison between the true survival
probability and the independent model.  More precisely, we shall prove that,
for every \(k>h\),
\[
    |P_k-(1-b)P_{k-1}|
    \le
    2hb^2+\frac{4q}{q-1}q^{-h}.
\]
This says that, conditionally on survival up to time \(k-1\), the new event
\(E_k\) behaves like an event of probability \(b\), up to an error caused by
the last \(h\) previous indices and by the remote-past boundary.

The reason for separating the past into a remote part and a short buffer is
as follows.  If one applied the one-point mixing estimate directly to
\(S_{k-1}\), the boundary of \(S_{k-1}\) would involve the recent dilates
\(a_{k-1},a_{k-2},\ldots\), and the resulting boundary error would not be
small compared with \(a_k\).  We therefore apply the mixing estimate only to
the remote survivor set \(R_k\), whose boundary is smaller than \(a_k\) by the
factor \(q^{-h}\).  The omitted buffer \(L_k\) is then restored by elementary
union bounds and by the exact pairwise identity
\[
    \PP(E_j\cap E_k)=b^2 .
\]
After proving the one-step estimate, we iterate it from \(h\) to \(N\).

    \medskip
    \noindent\textbf{1.  Boundary complexity of the remote past.}
    For a fixed $t$, let $(R_k)_t=\{x:(x,t)\in R_k\}$ be the $t$-fibre
    of $R_k$.  Since $B$ is a union of at most two intervals, it has at
    most four boundary points; each constraint $a_jx\notin t+B$ contributes
    at most $4a_j$ boundary points to $(R_k)_t$.  Hence
    \begin{equation}\label{eq:bdy}
        \#\partial(R_k)_t\le 4\sum_{j\le k-h}a_j
        \le 4\frac q{q-1}a_{k-h}
        \le \frac{4q}{q-1}\,q^{-h}a_k .
    \end{equation}

    \medskip
    \noindent\textbf{2.  Mixing over the remote past.}
    For fixed $t$, \cref{lem:onepoint} applied with $m=a_k$, survivor set
    $(R_k)_t$, and target $t+B$ gives the unnormalised estimate
    \[
        \left|
        \Leb_x\bigl((R_k)_t\cap\{x:a_kx-t\in B\}\bigr)
        -b\Leb_x((R_k)_t)
        \right|
        \le\frac{\#\partial(R_k)_t}{a_k}.
    \]
    This formulation also covers fibres of measure zero.  Integrating over
    $t\in\bbT$ yields the unconditional estimate
    \begin{equation}\label{eq:cond}
        \PP(R_k\cap E_k)=b\,\PP(R_k)+O_q(q^{-h}),
    \end{equation}
    where $|O_q(q^{-h})|\le \frac{4q}{q-1}q^{-h}$ by~\eqref{eq:bdy} and the
    explicit constant in \cref{lem:onepoint}.

    \medskip
    \noindent\textbf{3.  Exact pairwise independence.}
    For $j\neq k$, the map
    \[
        (x,t)\longmapsto (a_jx-t,\;a_kx-t)
    \]
    is a linear endomorphism of $\bbT^2$ with integer matrix
    $\bigl(\begin{smallmatrix} a_j & -1 \\ a_k & -1 \end{smallmatrix}\bigr)$.
    Its determinant is $a_k-a_j\neq0$, so the map is surjective and
    preserves Haar measure.  Consequently
    \begin{equation}\label{eq:pairwise}
        \PP(E_j\cap E_k)=\PP(E_j)\,\PP(E_k)=b^2 .
    \end{equation}
    Thus $E_j$ and $E_k$ are \emph{exactly pairwise independent} under
    $\PP$, despite the strong dependence among triples.

    \medskip
    \noindent\textbf{4.  Buffer estimates.}
    All bounds needed for the recursion are collected here.
    \begin{align}
        \PP(L_k^c)                &\le hb, \label{eq:buf1}\\
        \PP(R_k\cap L_k^c)        &\le hb, \label{eq:buf2}\\
        \PP(E_k\cap L_k^c)
           &\le\sum_{j=k-h+1}^{k-1}\PP(E_k\cap E_j)\le hb^2, \label{eq:buf3}\\
        \PP(R_k\cap E_k\cap L_k^c)&\le hb^2 . \label{eq:buf4}
    \end{align}
    Inequality~\eqref{eq:buf1} is the union bound over the $h-1$ events
    $E_j$ composing $L_k^c$; \eqref{eq:buf2} follows because
    $R_k\cap L_k^c\subseteq L_k^c$; \eqref{eq:buf3} uses the exact
    pairwise identity~\eqref{eq:pairwise}; \eqref{eq:buf4} is immediate
    from \eqref{eq:buf3}.

    From the disjoint decomposition in~\eqref{eq:id1} and~\eqref{eq:buf2},
    \begin{equation}\label{eq:remote-diff}
        \bigl|\PP(R_k)-\PP(R_k\cap L_k)\bigr|
        =\PP(R_k\cap L_k^c)\le hb .
    \end{equation}

    \medskip
    \noindent\textbf{5.  Recursion for the survival probability.}
    The quantity of interest is the conditional survival probability
    at step~$k$ given survival up to step~$k-1$:
    \[
        \PP(E_k^c\mid S_{k-1})
        =\frac{P_k}{P_{k-1}}
        =1-\frac{\PP(S_{k-1}\cap E_k)}{P_{k-1}} .
    \]
    We estimate the numerator $\PP(S_{k-1}\cap E_k)$.  Using
    $S_{k-1}=R_k\cap L_k$,
    \begin{align*}
        \PP(S_{k-1}\cap E_k)
        &=\PP(R_k\cap L_k\cap E_k) \\
        &=\PP(R_k\cap E_k)-\PP(R_k\cap E_k\cap L_k^c) \\
        &=b\,\PP(R_k)-\PP(R_k\cap E_k\cap L_k^c)+O_q(q^{-h})\\
        &=b\,\PP(R_k\cap L_k)
           +b\bigl(\PP(R_k)-\PP(R_k\cap L_k)\bigr)
           -\PP(R_k\cap E_k\cap L_k^c)+O_q(q^{-h}) .
    \end{align*}
    The second equality uses the disjoint decomposition
    \[
        R_k\cap E_k=(R_k\cap L_k\cap E_k)\cup(R_k\cap E_k\cap L_k^c).
    \]
    The third equality uses~\eqref{eq:cond}.  The first term is $bP_{k-1}$
    by~\eqref{eq:id1}.  The remaining three
    terms are controlled by~\eqref{eq:remote-diff}, \eqref{eq:buf4}, and
    the bound on $O_q(q^{-h})$:
    \[
        \bigl|\PP(S_{k-1}\cap E_k)-bP_{k-1}\bigr|
        \le b\cdot hb + hb^2 + \frac{4q}{q-1}q^{-h}
        =2hb^2+\frac{4q}{q-1}q^{-h}.
    \]
    Therefore the one-step recursion is \begin{equation}\label{eq:recur}
        |P_k-(1-b)P_{k-1}|
        =\bigl|\PP(S_{k-1}\cap E_k)-bP_{k-1}\bigr|
        \le 2hb^2+\frac{4q}{q-1}q^{-h},\qquad k>h.
    \end{equation}

    \medskip
    \noindent\textbf{6.  Base case and iteration.}
    For the first $h$ steps, the union bound gives
    $1-hb\le P_h\le1$, while the elementary inequality
    $(1-b)^h\ge 1-hb$ yields $1-hb\le(1-b)^h\le1$.  Hence
    $|P_h-(1-b)^h|\le hb$.  Iterating~\eqref{eq:recur} from $k=h+1$
    to $N$,
    \begin{align*}
        |P_N-(1-b)^N|
        &\le |P_h-(1-b)^h|+\sum_{k=h+1}^N|P_k-(1-b)P_{k-1}|\\
        &\le hb+2Nhb^2+\frac{4q}{q-1}Nq^{-h},
    \end{align*}
    which is exactly~\eqref{eq:escape}.  The second statement follows by
    substituting $b\le C_0(\log N)/N$ and $q^{-h}\le N^{-10}$.
\end{proof}

\begin{equation}\label{eq:ell-def}
    \ell_N(\tau)=\tau\frac{\log N}{N},\qquad
    M_N(\tau)=\bigl\lceil \ell_N(\tau)^{-1}\bigr\rceil,\qquad
    s_N(\tau)=M_N(\tau)^{-1}.
\end{equation}
Fix \(0<\tau<1/2\), and put
\[
    M=M_N(\tau),\qquad s=s_N(\tau)=M^{-1}.
\]
For \(\theta\in[0,s)\), define the shifted partition
\begin{equation}\label{eq:partition}
    \mathcal P_N^\theta(\tau)
    =
    \{I_j^\theta=\theta+[js,(j+1)s)\pmod1:\ 0\le j<M\}.
\end{equation}
Let
\begin{equation}\label{eq:K-def}
    K_N^{(\tau)}(x,\theta)
    =
    \#\{0\le j<M:\ I_j^\theta\cap A_N(x)=\varnothing\}.
\end{equation}

For the second-moment argument it is important to work on a fixed auxiliary
probability space.  We therefore write
\[
    \theta=us,\qquad u\in[0,1),
\]
and define
\begin{equation}\label{eq:Khat-def}
    \widehat K_N^{(\tau)}(x,u)
    =
    K_N^{(\tau)}(x,us).
\end{equation}
All expectations and probabilities below are taken with respect to the product
measure \(dx\,du\) on \(\bbT\times[0,1)\).

\begin{proposition}[Terminal empty cells]\label{prop:terminal}
    Let \(0<\tau<1/2\) and \(\alpha>1\).  Put
    \[
        N_m=\lfloor\alpha^m\rfloor .
    \]
    Then, for Lebesgue-almost every \(x\in\bbT\), for all sufficiently
    large \(m\) there exists \(\theta_m\in[0,s_{N_m}(\tau))\) such that
    \[
        K_{N_m}^{(\tau)}(x,\theta_m)\ge 2.
    \]
    Consequently,
    \[
        G_{N_m}(x)\ge s_{N_m}(\tau)
    \]
    for all sufficiently large \(m\).
\end{proposition}

\begin{proof}
    Fix \(N\), and abbreviate \(M=M_N(\tau)\), \(s=s_N(\tau)\).  For
    \(0\le j<M\), set
    \[
        X_j(x,u)
        =
        \mathbf 1\!\left(
        I_j^{us}\cap A_N(x)=\varnothing
        \right).
    \]
    Then
    \[
        \widehat K_N^{(\tau)}(x,u)=\sum_{j=0}^{M-1}X_j(x,u).
    \]

    \medskip
    \noindent\textbf{1. First moment.}
    For \(t\in\bbT\), define
    \[
        f_0(t)
        =
        \Leb_1\{x\in\bbT:\ a_nx-t\notin[0,s)
        \text{ for every }1\le n\le N\}.
    \]
    For a fixed \(j\), the change of variables \(t=us+js\) gives
    \[
        \E[X_j]
        =
        \frac1s\int_0^s f_0(\theta+js)\,d\theta .
    \]
   Notice that for fixed \(j\), the variable \(t=\theta+js\) is not uniformly distributed over the whole circle; it ranges only over the single cell \([js,(j+1)s)\). The average over the whole circle appears only after summing over all \(j\). Indeed,
    \begin{align}
        \E[\widehat K_N^{(\tau)}]
        &=
        \sum_{j=0}^{M-1}\E[X_j] \nonumber\\
        &=
        \frac1s\sum_{j=0}^{M-1}
        \int_0^s f_0(\theta+js)\,d\theta \nonumber\\
        &=
        \frac1s\int_{\bbT} f_0(t)\,dt \nonumber\\
        &=
        M\,P_N([0,s)).
        \label{eq:terminal-first-exact}
    \end{align}
    By the moving-target escape lemma,
    \[
        P_N([0,s))
        =
        (1-s)^N
        +
        O_{q,\tau}\!\left(\frac{(\log N)^3}{N}\right).
    \]
    Moreover,
    \[
        \frac{M(\log N)^3/N}{M(1-s)^N}
        =
        \frac{(\log N)^3}{N(1-s)^N}
        =
        N^{-(1-\tau)+o(1)}.
    \]
    Hence the relative error is polynomially small:
    \begin{align}
        \E[\widehat K_N^{(\tau)}]
        &=
        M(1-s)^N
        +
        O_{q,\tau}\!\left(M\frac{(\log N)^3}{N}\right) \nonumber\\
        &=
        M(1-s)^N
        \left(1+O_{q,\tau}\bigl(N^{-(1-\tau)+o(1)}\bigr)\right).
        \label{eq:terminal-EK}
    \end{align}
    Since
    \[
        M=\frac1s=(1+o(1))\frac{N}{\tau\log N},
        \qquad
        Ns=(1+o(1))\tau\log N,
    \]
    we get
    \begin{equation}\label{eq:terminal-EK-asymp}
        \E[\widehat K_N^{(\tau)}]
        =
        \frac{N^{1-\tau+o(1)}}{\tau\log N}
        \longrightarrow\infty .
    \end{equation}

    \medskip
    \noindent\textbf{2. Second moment.}
    For \(0\le r<M\), define
    \[
        B_r=
        \begin{cases}
            [0,s), & r=0,\\[3pt]
            [0,s)\cup[rs,(r+1)s), & 1\le r<M.
        \end{cases}
    \]
    Thus \(|B_0|=s\), while \(|B_r|=2s\) for \(1\le r<M\).  For
    \(t\in\bbT\), put
    \[
        f_r(t)
        =
        \Leb_1\{x\in\bbT:\ a_nx-t\notin B_r
        \text{ for every }1\le n\le N\}.
    \]

    In the following, \(j+r\) is always interpreted modulo \(M\).  For fixed
    \(j\) and \(r\), the joint event \(X_j=X_{j+r}=1\) says that all the
    points \(a_nx\) avoid the two cells
    \[
        us+[js,(j+1)s)
        \quad\text{and}\quad
        us+[(j+r)s,(j+r+1)s)
    \]
    modulo one.  With \(t=us+js\), this is precisely the moving-target
    avoidance event with target \(B_r\).  Therefore
    \[
        \E[X_jX_{j+r}]
        =
        \frac1s\int_0^s f_r(\theta+js)\,d\theta .
    \]
    Summing over \(j\), the same averaging mechanism as in the first moment
    gives
    \[
        \sum_{j=0}^{M-1}\E[X_jX_{j+r}]
        =
        \frac1s\int_{\bbT} f_r(t)\,dt
        =
        M\,P_N(B_r).
    \]
    Hence
    \begin{equation}\label{eq:terminal-second-exact}
        \E[(\widehat K_N^{(\tau)})^2]
        =
        M\sum_{r=0}^{M-1}P_N(B_r).
    \end{equation}

    Applying the moving-target escape lemma to \(B_0\) and to the \(B_r\)'s
    with \(1\le r<M\), we obtain uniformly in \(r\)
    \[
        P_N(B_0)
        =
        (1-s)^N
        +
        O_{q,\tau}\!\left(\frac{(\log N)^3}{N}\right),
    \]
    and
    \[
        P_N(B_r)
        =
        (1-2s)^N
        +
        O_{q,\tau}\!\left(\frac{(\log N)^3}{N}\right),
        \qquad 1\le r<M.
    \]
    Therefore
    \begin{equation}\label{eq:terminal-EK2}
        \E[(\widehat K_N^{(\tau)})^2]
        \le
        M(1-s)^N
        +
        M^2(1-2s)^N
        +
        O_{q,\tau}\!\left(M^2\frac{(\log N)^3}{N}\right).
    \end{equation}

    \medskip
    \noindent\textbf{3. Variance estimate.}
    We compare \eqref{eq:terminal-EK2} with
    \((\E[\widehat K_N^{(\tau)}])^2\).  From the quantitative form of
    \eqref{eq:terminal-EK},
    \[
        (\E[\widehat K_N^{(\tau)}])^2
        =
        M^2(1-s)^{2N}
        \left(1+O_{q,\tau}\bigl(N^{-(1-\tau)+o(1)}\bigr)\right).
    \]
    The diagonal contribution satisfies
    \[
        \frac{M(1-s)^N}{M^2(1-s)^{2N}}
        =
        \frac1{M(1-s)^N}
        =
        N^{-(1-\tau)+o(1)}.
    \]
    For the main off-diagonal contribution, Taylor expansion gives
    \[
        \frac{(1-2s)^N}{(1-s)^{2N}}
        =
        \exp\!\left(
        N\log(1-2s)-2N\log(1-s)
        \right)
        =
        1+O_\tau\!\left(\frac{(\log N)^2}{N}\right).
    \]
    Finally, the error term from the escape lemma gives
    \[
        \frac{M^2(\log N)^3/N}
             {M^2(1-s)^{2N}}
        =
        \frac{(\log N)^3}{N(1-s)^{2N}}
        =
        N^{-(1-2\tau)+o(1)}.
    \]
    Since \(\tau<1/2\), we may choose
    \[
        \gamma=\frac12(1-2\tau)>0.
    \]
    Combining the preceding estimates yields
    \begin{equation}\label{eq:terminal-ratio}
        \frac{\E[(\widehat K_N^{(\tau)})^2]}
             {(\E[\widehat K_N^{(\tau)}])^2}
        \le
        1+O_{q,\tau}(N^{-\gamma}).
    \end{equation}
    Consequently,
    \begin{equation}\label{eq:terminal-variance}
        \Var(\widehat K_N^{(\tau)})
        \le
        O_{q,\tau}(N^{-\gamma})
        \bigl(\E[\widehat K_N^{(\tau)}]\bigr)^2.
    \end{equation}

    \medskip
    \noindent\textbf{4. Chebyshev and Borel--Cantelli.}
    By \eqref{eq:terminal-EK-asymp}, for all sufficiently large \(N\),
    \[
        \E[\widehat K_N^{(\tau)}]\ge4.
    \]
    Hence
    \[
        \widehat K_N^{(\tau)}<2
        \quad\Longrightarrow\quad
        \left|
        \widehat K_N^{(\tau)}
        -
        \E[\widehat K_N^{(\tau)}]
        \right|
        \ge
        \frac12\E[\widehat K_N^{(\tau)}].
    \]
    Chebyshev's inequality and \eqref{eq:terminal-variance} give
    \begin{equation}\label{eq:terminal-chebyshev}
        \PP\{\widehat K_N^{(\tau)}<2\}
        \le
        O_{q,\tau}(N^{-\gamma}).
    \end{equation}

    Now take \(N_m=\lfloor\alpha^m\rfloor\) for any \(\alpha>1\).  Since
    \[
        \sum_{m=1}^\infty N_m^{-\gamma}<\infty,
    \]
    the Borel--Cantelli lemma applied on the fixed probability space
    \(\bbT\times[0,1)\) implies that for almost every pair \((x,u)\),
    \[
        \widehat K_{N_m}^{(\tau)}(x,u)\ge2
    \]
    for all sufficiently large \(m\).

    By Fubini, for Lebesgue-almost every \(x\), this conclusion holds for
    almost every \(u\in[0,1)\).  Fix such an \(x\), and choose one admissible
    \(u\).  For all sufficiently large \(m\), set
    \[
        \theta_m=u\,s_{N_m}(\tau).
    \]
    Then
    \[
        K_{N_m}^{(\tau)}(x,\theta_m)
        =
        \widehat K_{N_m}^{(\tau)}(x,u)
        \ge2.
    \]

    In particular, at least one cell of
    \(\mathcal P_{N_m}^{\theta_m}(\tau)\) is empty, and every cell has circular
    length \(s_{N_m}(\tau)\).  Therefore
    \[
        G_{N_m}(x)\ge s_{N_m}(\tau)
    \]
    for all sufficiently large \(m\).  Moreover, since at most one cell of a
    shifted circular partition wraps around \(0\), the stronger conclusion
    \(K_{N_m}^{(\tau)}(x,\theta_m)\ge2\) also guarantees the existence of a
    non-wrapping empty interval if such an interval is needed later.
\end{proof}

\begin{proof}[Proof of~\eqref{eq:main-liminf}]
    Let \(0<c<1/2\).  Choose \(0<\tau<1/2\) and \(\alpha>1\) such that
    \[
        c<\frac{\tau}{\alpha}.
    \]
    By \cref{prop:terminal}, for Lebesgue-almost every \(x\),
    \[
        G_{N_m}(x)\ge s_{N_m}(\tau)
        =
        (1+o(1))\tau\frac{\log N_m}{N_m}
    \]
    for all sufficiently large \(m\), where \(N_m=\lfloor\alpha^m\rfloor\).

    Let \(N_m\le N<N_{m+1}\).  Since adding points can only decrease the
    maximal gap,
    \[
        G_N(x)\ge G_{N_{m+1}}(x).
    \]
    Therefore, for all sufficiently large \(m\),
    \[
        G_N(x)
        \ge
        (1+o(1))\tau\frac{\log N_{m+1}}{N_{m+1}}.
    \]
    Multiplying by \(N/\log N\), and using
    \[
        \frac{N}{N_{m+1}}\ge \frac1\alpha+o(1),
        \qquad
        \frac{\log N_{m+1}}{\log N}=1+o(1),
    \]
    we obtain
    \[
        \frac{N\,G_N(x)}{\log N}
        \ge
        \frac{\tau}{\alpha}+o(1).
    \]
    Hence
    \[
        \liminf_{N\to\infty}
        \frac{N\,G_N(x)}{\log N}
        \ge
        \frac{\tau}{\alpha}
        >
        c.
    \]
    Letting \(c\uparrow1/2\) through a countable sequence gives
    \[
        \liminf_{N\to\infty}
        \frac{N\,G_N(x)}{\log N}
        \ge
        \frac12
    \]
    for Lebesgue-almost every \(x\).
\end{proof}

\section{The upper bound with local correlation constant}\label{sec:upper}
In this section, we prove the refined upper bound in~\eqref{eq:main-limsup}.
   Since \(a_n/a_m\ge q^{n-m}\),
\begin{equation}\label{eq:Gamma-bound}
    \Gamma(p,L)
    \le
    \frac1L\sum_{r=1}^{L}(L-r+1)q^{-r}
    < \frac1{q-1}.
\end{equation}
\paragraph{Rounded Geometric progressions.}
If \(a_n=\lceil Ca^n \rceil\) with $a>1$ and $C>(a-1)^{-1}$, then $\Gamma=(a-1)^{-1}$, so \cref{thm:one-d} gives  
\[
    \limsup_{N\to\infty}\frac{NG_N(x)}{\log N}
    \le \frac{a+1}{a-1}.
\]
for Lebesgue-almost every \(x\). If, in addition,  \( a,C\in \bbZ\),  then \cref{thm:divisible} sharpens it to
\[
    \lim_{N\to\infty}\frac{NG_N(x)}{\log N}=1 \: \text{ a.e.}
\]

\paragraph{Alternating ratios.}
Suppose that for some integers $a,b>1$,
\begin{equation}\label{eq:ex-alternating}
    \frac{a_{n+1}}{a_n}
    =
    \begin{cases}
        a, & n\text{ odd},\\[2pt]
        b, & n\text{ even},
    \end{cases}
\end{equation}
Then
\[
    \Gamma
    =
    \frac{1/a+1/b+2/(ab)}{2(1-1/(ab))}.
\]
For example, if \(a=2\) and \(b=50\), then
\[
    \Gamma=\frac{3}{11},
    \qquad
    1+2\Gamma=\frac{17}{11}\approx1.545.
\]
The worst-ratio estimate with \(a=2\) would only give \(1+2/(a-1)=3\).

\medskip
Fix a target interval $J\subset\bbT$ of length
\begin{equation}\label{eq:s-def}
    s=\tau\frac{\log N}{N}.
\end{equation}
Let $\beta>0$, and set
\begin{equation}\label{eq:Lh-def}
    L=\bigl\lceil\frac{\beta}{s}\bigr\rceil,\qquad
    h=\bigl\lceil\frac{10\log N}{\log q}\bigr\rceil.
\end{equation}
Define blocks
\begin{equation}\label{eq:blocks-def}
    p_u=1+(u-1)(L+h),\qquad
    \Delta_u=\{p_u,p_u+1,\ldots,p_u+L-1\},
\end{equation}
and let $T=T_N$ be the largest $u$ with $\Delta_u\subset[1,N]$.  Since
$h=O_q(\log N)$ and $L\sim\beta/s$,
\begin{equation}\label{eq:T-upper}
    T=(1+o(1))\,\frac{\tau}{\beta}\,\log N.
\end{equation}

Define survivor sets
\[
    \Omega_0(J)=\bbT,
\]
and
\[
    \Omega_u(J)
    =\{x\in\Omega_{u-1}(J): a_nx\notin J\text{ for every }n\in\Delta_u\}.
\]
Set
\[
    S_u^J(x)=\sum_{n\in\Delta_u}\ind_J(a_nx),\qquad
    \mu=Ls.
\]


\begin{lemma}[Boundary complexity]\label{lem:boundary}
    For $F=\Omega_{u-1}(J)$,
    \[
        \#\partial F\le \frac{2}{q-1}\,q^{-h}a_{p_u}
        \le \frac{2}{q-1}\,N^{-10}a_{p_u}.
    \]
    For $u=1$, $\#\partial F=0$.
\end{lemma}

\begin{proof}
    For $u=1$, $F=\bbT$.  For $u\ge2$, the boundary of $F$ is contained in
    the union of the boundaries of the sets
    \[
        \{x: a_nx\in J\},\qquad n\le p_u-h-1.
    \]
    Each such set has at most $2a_n$ boundary points.  Hence
    \begin{align*}
        \#\partial F &\le 2\sum_{n\le p_u-h-1}a_n
               \le 2\frac q{q-1}a_{p_u-h-1}
               \le \frac2{q-1}q^{-h}a_{p_u}.
    \end{align*}
    With $h=\lceil10\log N/\log q\rceil$, we have $q^{-h}\le N^{-10}$.
\end{proof}

\begin{lemma}[Block moments]\label{lem:block-1d}
    Let $F=\Omega_{u-1}(J)$, and assume $\Leb_1(F)\ge N^{-4}$.  For every
    $\varepsilon>0$, for all sufficiently large $N$,
    \begin{align}
        \int_F S_u^J(x)\,dx
        &= \mu \Leb_1(F)+o(\Leb_1(F)), \label{eq:block-first}\\[4pt]
        \int_F(S_u^J(x))^2\,dx
        &\le
        \Leb_1(F)\bigl(\mu^2+\mu+2(\Gamma+\varepsilon)\mu\bigr)+o(\Leb_1(F)), \label{eq:block-second}
    \end{align}
    uniformly in \(u\) and \(J\).  The \(o(\Leb_1(F))\) terms are controlled by
    \(C_q N^{-10}\le C_q N^{-6}\Leb_1(F)\), so the implicit constant does not depend
    on the value of~\(\Leb_1(F)\) in the range \(\Leb_1(F)\ge N^{-4}\).
\end{lemma}

\begin{proof}
    The first moment follows from \cref{lem:onepoint},
    \cref{lem:boundary}, and
    \[
        \sum_{n\in\Delta_u}\frac1{a_n}\le\frac{q}{(q-1)a_{p_u}}.
    \]
    The error is $O_q(N^{-10})=o(\Leb_1(F))$.

    For the second moment,
    \[
        (S_u^J)^2
        = \sum_{n\in\Delta_u}\ind_J(a_nx)
          + 2\sum_{\substack{n,m\in\Delta_u\\ m<n}}
            \ind_J(a_mx)\,\ind_J(a_nx).
    \]
    The diagonal part contributes at most $\mu\Leb_1(F)+o(\Leb_1(F))$.
    \Cref{lem:twopoint} gives, for $m<n$,
    \[
        \int_F\ind_J(a_mx)\,\ind_J(a_nx)\,dx
        \le \Leb_1(F)\Bigl(s^2+s\frac{a_m}{a_n}\Bigr)
           + \#\partial F\frac{s}{a_m}.
    \]
    The $s^2$-part contributes at most $\mu^2\Leb_1(F)$.  The boundary contribution
    is $o(\Leb_1(F))$.

    It remains to estimate the double sum over $m,n$
    \[
        2s\Leb_1(F)\sum_{m<n\in\Delta_u}\frac{a_m}{a_n}.
    \]
    Given $\varepsilon>0$, \eqref{eq:Gamma-main} implies that for sufficiently large $L$, we have 
    $\sup_{p \ge 1} \Gamma(p,L)< \Gamma+\varepsilon$.
    Since $L=\lceil \frac{\beta}{s} \rceil \to\infty$ as $N \to \infty$, it follows  that for sufficiently large $N$,
    \[  \sum_{m<n\in\Delta_u}\frac{a_m}{a_n}
        \le L(\Gamma+\varepsilon)
    \]
    for all $u \ge 1$.  Hence the off-diagonal
    lacunary contribution is at most
    \[
        2sL(\Gamma+\varepsilon)\Leb_1(F)=2(\Gamma+\varepsilon)\,\mu\,\Leb_1(F).
    \]
\end{proof}

\begin{proposition}[Fixed interval no-hit estimate]\label{prop:nohit-1d}
    For every fixed $\beta>0$, $\tau>0$, and $\varepsilon>0$,
    \begin{equation}\label{eq:nohit-1d}
        \Leb_1\{x: a_nx\notin J\text{ for all }1\le n\le N\}
        \le
        N^{-\tau/(\beta+1+2(\Gamma+\varepsilon))+o(1)}
        + O(N^{-4})
    \end{equation}
    uniformly over all intervals $J$ of length $s=\tau(\log N)/N$.
\end{proposition}

\begin{proof}
    If $a_nx\notin J$ for all $1\le n\le N$, then $x\in\Omega_T(J)$.
    The survivor sets are nested.  If, for some $1\le u\le T$,
    \(\Leb_1(\Omega_{u-1}(J))<N^{-4}\), then automatically
    \(\Leb_1(\Omega_T(J))<N^{-4}\), and there is nothing more to prove.

    Otherwise \(\Leb_1(\Omega_{u-1}(J))\ge N^{-4}\) for every
    \(1\le u\le T\), so \cref{lem:block-1d} applies to every active block.
    With \(F=\Omega_{u-1}(J)\), the Paley--Zygmund inequality gives
    \begin{align*}
        \Leb_1\{x\in F:S_u^J(x)>0\}
        &\ge
        \frac{\bigl(\int_F S_u^J\,dx\bigr)^2}
             {\int_F(S_u^J)^2\,dx}\\[4pt]
        &\ge
        \left(
        \frac{\mu}{\mu+1+2(\Gamma+\varepsilon)+o(1)}
        \right)\Leb_1(F),
    \end{align*}
    uniformly in \(u\).  The error here collects the
    \(o(\Leb_1(F))\) terms from \eqref{eq:block-first} and
    \eqref{eq:block-second}.  Since \(\mu\to\beta\), it follows that
    \[
        \Leb_1(\Omega_u(J))
        \le
        \left(
        1-\frac{\beta}{\beta+1+2(\Gamma+\varepsilon)}+o(1)
        \right)
        \Leb_1(\Omega_{u-1}(J))
    \]
    for every active block.  Iterating this inequality and using
    \eqref{eq:T-upper} gives
    \[
        \Leb_1(\Omega_T(J))
        \le
        \exp\!\left(
        -\left(\frac{\beta}{\beta+1+2(\Gamma+\varepsilon)}+o(1)\right)
        \frac{\tau}{\beta}\log N
        \right)
        +O(N^{-4}),
    \]
    which is~\eqref{eq:nohit-1d}.
\end{proof}

\begin{proof}[Proof of~\eqref{eq:main-limsup}]
    Fix $\rho>0$.  For each $N$, choose a family $\mathcal J_N$ of intervals
    of length $s=\tau(\log N)/N$ whose left endpoints form a mesh of size at
    most $\rho s/2$.  Then
    \[
        \#\mathcal J_N=O_{\rho,\tau}\!\Bigl(\frac N{\log N}\Bigr),
    \]
    and if $G_N(x)>(1+\rho)s$, then some $J\in\mathcal J_N$ is empty.  By
    \cref{prop:nohit-1d},
    \begin{align*}
        \Leb_1\{x: G_N(x)>(1+\rho)s\}
        &\le
        O_{\rho,\tau}\!\Bigl(\frac N{\log N}\Bigr)
        \Bigl(
        N^{-\tau/(\beta+1+2(\Gamma+\varepsilon))+o(1)}
        + O(N^{-4})
        \Bigr).
    \end{align*}
    Along a geometric sequence \(N_m=\lfloor\alpha^m\rfloor\), the above
    probabilities are summable whenever
    \[
        \tau>\beta+1+2(\Gamma+\varepsilon).
    \]
    Thus Borel--Cantelli implies that, almost surely,
    \[
        G_{N_m}(x)\le (1+\rho)\tau\frac{\log N_m}{N_m}
    \]
    for all sufficiently large \(m\). Hence
    \[
        \limsup_{m\to\infty}
        \frac{N_mG_{N_m}(x)}{\log N_m}
        \le (1+\rho)\tau.
    \]
    Letting
    \[
        \tau\downarrow\beta+1+2(\Gamma+\varepsilon)
    \]
    and then \(\rho\downarrow0\), we obtain
    \[
        \limsup_{m\to\infty}
        \frac{N_mG_{N_m}(x)}{\log N_m}
        \le
        \beta+1+2(\Gamma+\varepsilon).
    \]
    For $N_m\le N<N_{m+1}$, we have $G_N(x)\le G_{N_m}(x)$ (adding points
    reduces gaps), hence
    \[
        \frac{NG_N(x)}{\log N}
        \le \alpha\,\frac{N_m G_{N_m}(x)}{\log N_m}\cdot\frac{\log N_m}{\log N}
        \le (\alpha+o(1))\,\frac{N_m G_{N_m}(x)}{\log N_m}.
    \]
    Taking $\limsup_{N\to\infty}$ and then $\alpha\downarrow1$ yields the
    same bound.  Let $\beta\downarrow0$ and $\varepsilon\downarrow0$ through
    countable sequences to obtain \(\limsup_{N\to\infty} NG_N(x)/\log N
    \le 1+2\Gamma\).
\end{proof}

\section{Divisibility chains: the sharp constant}\label{sec:divisible}

Throughout this section assume
\begin{equation}\label{eq:divisibility}
    a_n\mid a_{n+1},\qquad 0<a_n<a_{n+1},\qquad n\ge1.
\end{equation}
Thus every quotient \(a_{n+1}/a_n\) is an integer at least~\(2\).  Replacing
\(x\) by \(a_1x\) and \(a_n\) by \(a_n/a_1\) does not change the
Lebesgue-almost-sure assertion, so we may and shall assume that \(a_1=1\).
For an integer \(r\ge2\), write
\[
    T_r y=ry\pmod1.
\]
We first prove the sharp upper bound by separating intervals with short
returns.  We then use the mixed-radix expansion associated with
\eqref{eq:divisibility} to prove the matching lower bound.

\subsection{The sharp upper bound}

For each sufficiently large \(N\), put
\begin{equation}\label{eq:div-R}
    R=R_N=\bigl\lfloor(\log N)^4\bigr\rfloor.
\end{equation}
An interval \(J\subset\bbT\) is called \emph{\(R\)-regular} if
\begin{equation}\label{eq:R-regular}
    J\cap T_r^{-1}J=\varnothing
    \qquad\text{for every integer }2\le r\le R.
\end{equation}

\begin{lemma}[Counting intervals with a short return]\label{lem:short-return-count}
    Let \(s=o(R^{-1})\), and let \(\mathcal J\) be a circular arithmetic mesh
    of intervals of length \(s\), with consecutive left endpoints separated
    by a number in \([\rho s/4,\rho s/2]\), where \(\rho>0\) is fixed.  Then
    the number of intervals in \(\mathcal J\) that are not \(R\)-regular is
    \(O_\rho(R^2)\).
\end{lemma}

\begin{proof}
    Fix \(2\le r\le R\).  If \(J\cap T_r^{-1}J\neq\varnothing\), there is a
    point \(y\in J\) such that \(ry\in J\pmod1\).  Since \(J\) has circular
    length \(s<1/2\),
    \[
        \|(r-1)y\|\le s.
    \]
    Hence \(y\) lies within \(s/(r-1)\) of one of the \(r-1\) fixed points
    \(k/(r-1)\) of \(T_r\).  The left endpoint of \(J\) is therefore within
    \(2s\) of one of these fixed points.  The mesh contains only
    \(O_\rho(1)\) such intervals per fixed point, so there are
    \(O_\rho(r)\) intervals with an \(r\)-return.  Summing over
    \(2\le r\le R\) proves the lemma.
\end{proof}

\begin{proposition}[No-hit estimate for regular intervals]
\label{prop:regular-nohit}
    Fix \(\beta>0\) and \(\tau>0\), and let
    \(s=\tau(\log N)/N\).  Uniformly over all \(R_N\)-regular intervals
    \(J\subset\bbT\) of length \(s\),
    \begin{equation}\label{eq:regular-nohit}
        \Leb_1\{x:a_nx\notin J\text{ for }1\le n\le N\}
        \le
        N^{-\tau/(\beta+1+4/R_N)+o(1)}+O(N^{-4}).
    \end{equation}
\end{proposition}

\begin{proof}
    Use the block construction of \cref{sec:upper}, with
    \(L=\lceil\beta/s\rceil\), buffer length
    \(h=\lceil10\log N/\log2\rceil\), and the corresponding survivor sets
    \(\Omega_u(J)\).  Since \eqref{eq:divisibility} implies
    \(a_{n+1}\ge2a_n\), the boundary estimate of \cref{lem:boundary} applies
    with \(q=2\).

    Let \(F=\Omega_{u-1}(J)\), assume \(\Leb_1(F)\ge N^{-4}\), and put
    \(S_u^J=\sum_{n\in\Delta_u}\ind_J(a_nx)\) and \(\mu=Ls\).  The first
    moment is
    \[
        \int_F S_u^J\,dx=\mu\Leb_1(F)+o(\Leb_1(F)).
    \]
    For \(m<n\), the quotient
    \[
        r_{m,n}=\frac{a_n}{a_m}
    \]
    is an integer.  If \(r_{m,n}\le R\), regularity gives
    \(\ind_J(a_mx)\ind_J(a_nx)=0\) identically.  If \(r_{m,n}>R\),
    \cref{lem:twopoint} gives
    \[
        \int_F\ind_J(a_mx)\ind_J(a_nx)\,dx
        \le
        \Leb_1(F)\left(s^2+\frac{s}{r_{m,n}}\right)
        +\#\partial F\frac{s}{a_m}.
    \]
    For fixed \(m\), the successive quotients \(r_{m,n}\) at least double,
    and therefore
    \[
        \sum_{\substack{n>m\\r_{m,n}>R}}\frac1{r_{m,n}}
        \le\frac2R.
    \]
    Summing over the block, and estimating the boundary terms exactly as in
    \cref{lem:block-1d}, yields
    \begin{equation}\label{eq:regular-second}
        \int_F(S_u^J)^2\,dx
        \le
        \Leb_1(F)\left(\mu^2+\mu+\frac{4\mu}{R}\right)
        +o(\Leb_1(F)).
    \end{equation}
    Paley--Zygmund therefore removes, from each active block, the proportion
    \[
        \frac{\mu}{\mu+1+4/R+o(1)}.
    \]
    The survivor sets are nested, so either one of them has measure below
    \(N^{-4}\), or the contraction applies to every block.  Since the number
    of complete blocks is
    \((1+o(1))(\tau/\beta)\log N\), iteration gives
    \[
        \Leb_1(\Omega_T(J))
        \le
        \exp\left(
        -\left(\frac{\beta}{\beta+1+4/R}+o(1)\right)
        \frac{\tau}{\beta}\log N
        \right)+O(N^{-4}),
    \]
    which is \eqref{eq:regular-nohit}.
\end{proof}

\begin{proof}[Sharp upper bound in \cref{thm:divisible}]
    Fix \(\tau>1\), \(\rho>0\), and choose \(\beta>0\) so small that
    \(\tau>1+\beta\).  For each \(N\), let \(\mathcal J_N\) be a circular
    arithmetic mesh of intervals of length \(s=\tau(\log N)/N\), with mesh
    spacing in \([\rho s/4,\rho s/2]\).  Then
    \[
        \#\mathcal J_N=O_{\rho,\tau}\left(\frac N{\log N}\right),
    \]
    and if \(G_N(x)>(1+\rho)s\), some member of \(\mathcal J_N\) is empty.

    For the regular members, \cref{prop:regular-nohit} and a union bound give
    \[
        O_{\rho,\tau}\left(\frac N{\log N}\right)
        N^{-\tau/(\beta+1+4/R_N)+o(1)}.
    \]
    This is summable along every geometric sequence for all sufficiently
    large \(N\).  By \cref{lem:short-return-count}, only \(O_\rho(R_N^2)\)
    mesh intervals are exceptional.  Since divisibility implies lacunarity
    with \(q=2\), \eqref{eq:Gamma-bound} gives \(\Gamma\le1\).  Applying
    \cref{prop:nohit-1d} with, say, \(\varepsilon=1\), the exceptional
    contribution is at most
    \[
        O_\rho(R_N^2)
        \left(N^{-\tau/(\beta+5)+o(1)}+O(N^{-4})\right),
    \]
    which is also summable along geometric sequences.

    Borel--Cantelli, followed by the interpolation argument used at the end
    of \cref{sec:upper}, gives
    \[
        \limsup_{N\to\infty}\frac{NG_N(x)}{\log N}\le(1+\rho)\tau
    \]
    for almost every \(x\).  Letting \(\rho\downarrow0\) and
    \(\tau\downarrow1\) through countable sequences proves
    \begin{equation}\label{eq:div-upper}
        \limsup_{N\to\infty}\frac{NG_N(x)}{\log N}\le1.
    \end{equation}
\end{proof}

\subsection{Mixed-radix digits and local dependence}

Put
\[
    q_n=\frac{a_{n+1}}{a_n}\in\{2,3,\ldots\}.
\]
Outside the countable set of expansion endpoints, every \(x\in[0,1)\) has a
unique mixed-radix expansion
\begin{equation}\label{eq:mixed-radix}
    x=\sum_{j=1}^\infty\frac{\xi_j}{a_{j+1}},
    \qquad 0\le\xi_j<q_j.
\end{equation}
Under Lebesgue measure the digits \((\xi_j)_{j\ge1}\) are independent and
\(\xi_j\) is uniform on \(\{0,\ldots,q_j-1\}\).  Indeed, prescribing the
first \(k\) digits selects one interval of length \(a_{k+1}^{-1}\).  Moreover,
\begin{equation}\label{eq:mixed-tail}
    a_nx\pmod1
    =
    \sum_{j=n}^\infty
    \frac{\xi_j}{q_nq_{n+1}\cdots q_j}.
\end{equation}

We use Suen's correlation inequality in the form recorded in
\cite{JansonSuen}, together with the standard Lov\'asz local lemma \cite{AS}.

\begin{definition}[Dependency graph]\label{def:dependency-graph}
Let \(A_1,\ldots,A_N\) be events on a probability space
\((\Omega,\mathcal F,\PP)\).  A simple undirected graph
\(\mathcal G=([N],E)\), where \([N]=\{1,\ldots,N\}\), is called a
dependency graph for the family \((A_i)_{i=1}^N\) if the following
condition holds.

Whenever \(I,J\subset [N]\) are disjoint sets such that no edge of
\(\mathcal G\) joins a vertex of \(I\) to a vertex of \(J\), the
sigma-algebras
\[
    \sigma(\mathbf 1_{A_i}:i\in I)
    \qquad\text{and}\qquad
    \sigma(\mathbf 1_{A_j}:j\in J)
\]
are independent.

We write \(i\sim j\) if \(\{i,j\}\in E\), and denote the maximum degree by
\[
    D(\mathcal G)
    =
    \max_{1\le i\le N}\#\{j\in[N]:j\sim i\}.
\]
\end{definition}

\begin{lemma}[Two-sided local-dependence estimate]\label{lem:local-dependence}
Let $A_1,\ldots,A_N$ be events with a dependency graph of maximum degree $D$.
Assume
\[
|\PP(A_i)-p|\le \varepsilon,\qquad p_*:=p+\varepsilon,
\qquad 4Dp_* \le \frac12.
\]
Define the edge correlation term
\[
\Delta := \sum_{\{i,j\}\in E} \PP(A_i \cap A_j),
\]
where the sum is over unordered edges of the dependency graph. Then
\[
\PP\Big(\bigcap_{i=1}^N A_i^c\Big)
=
\exp\left(-Np
+
O\big(N\varepsilon+\Delta + NDp_*^2 + Np_*^2\big)\right).
\]
\end{lemma}
\begin{proof}
Let
\[
X=\sum_{i=1}^N \mathbf{1}_{A_i}.
\]
Then
\[
\PP\Big(\bigcap_i A_i^c\Big)=\PP(X=0).
\]
By Suen's inequality,
\[
\PP(X=0)\le \exp\big(-\sum_{i=1}^N \PP(A_i)+ O(\Delta)\big)=\exp\big(-Np + O(N\varepsilon+\Delta) \big)\, .
\]

For the lower bound, define \(z\in(0,1)\) implicitly by
\[
z(1-z)^D = p_* .
\]
Since \(Dp_*\) is sufficiently small, a Taylor expansion yields
\[
z = p_* + O(Dp_*^2).
\]

We apply the symmetric Lovász local lemma in the standard form: if
\[
\PP(A_i)\le z(1-z)^D \quad \text{for all } i,
\]
then
\[
\PP\Big(\bigcap_{i=1}^N A_i^c\Big)\ge (1-z)^N.
\]
Hence
\[
\PP\Big(\bigcap_i A_i^c\Big)
\ge \exp\big(N\log(1-z)\big).
\]

Using \(\log(1-z)=-z+O(z^2)\), we obtain
\[
N\log(1-z)
=
-Nz + O(Nz^2).
\]
Substituting \(z=p_*+O(Dp_*^2)\) gives
\[
N\log(1-z)
=
-Np_* + O(NDp_*^2 + Np_*^2).
\]

Since
\[
p_*=p+\varepsilon,
\]
we conclude
\[
\PP\Big(\bigcap_i A_i^c\Big)
\ge
\exp\big(-Np + O(N\varepsilon+NDp_*^2 + Np_*^2)\big).
\]

Combining the upper and lower bounds yields
\[
 \PP\Big(\bigcap_i A_i^c\Big)
=\exp(
-Np
+
O\big(N\varepsilon+\Delta + NDp_*^2 + Np_*^2\big)),
\]
as required.
\end{proof}
\begin{lemma}[Integer overlap]\label{lem:integer-overlap}
    Let \(U,V\subset\bbT\) each be a union of at most two intervals.  For
    every integer \(r\ge1\),
    \begin{equation}\label{eq:integer-overlap}
        \Leb_1(U\cap T_r^{-1}V)
        \le \Leb_1(U)\Leb_1(V)+\frac{2\Leb_1(V)}r.
    \end{equation}
\end{lemma}

\begin{proof}
    It is enough to treat one interval \(C\) of \(U\).  Lift \(C\) to an
    interval in \(\bbR\), and change variables \(z=ry\).  If \(|C|=c\) and
    \(|V|=v\), then every interval \(K\subset\bbR\) of length \(L\) satisfies
    \[
        |K\cap(V+\bbZ)|\le Lv+v.
    \]
    Indeed, the complete unit intervals contribute exactly \(v\) each, and
    the remaining interval of length less than one contributes at most \(v\).
    Therefore
    \[
        \Leb_1(C\cap T_r^{-1}V)\le cv+\frac vr.
    \]
    Summing over at most two components of \(U\) proves the result.
\end{proof}

We next construct many candidate gaps with no short return, including no
short return between any two candidates.

\begin{lemma}[Separated interval family]\label{lem:separated-family}
    Let \(0<\tau<1\), define \(M=M_N(\tau)\) and \(s=M^{-1}\) as in
    \eqref{eq:ell-def}, and put
    \[
        R=\lfloor(\log N)^4\rfloor,
        \qquad Q=N^{20},
        \qquad \eta=2Q^{-1}.
    \]
    For all sufficiently large \(N\), there is a subfamily
    \(\mathcal C_N\) of the partition
    \[
        \mathcal P_N=\{[js,(j+1)s):0\le j<M\}
    \]
    such that
    \begin{equation}\label{eq:family-size}
        \#\mathcal C_N\gg\frac{M}{R^2},
    \end{equation}
    and the following holds.  If \(H\) is either one member of
    \(\mathcal C_N\) or the union of two distinct members, and
    \(H^{(\eta)}\) is the union of their closed \(\eta\)-neighbourhoods, then
    \begin{equation}\label{eq:family-separation}
        H^{(\eta)}\cap T_r^{-1}H^{(\eta)}=\varnothing,
        \qquad 2\le r\le R.
    \end{equation}
\end{lemma}

\begin{proof}
    Call a partition interval \(I\) bad for \(r\) if
    \(I^{(\eta)}\cap T_r^{-1}I^{(\eta)}\neq\varnothing\).  As in the proof of
    \cref{lem:short-return-count}, such an interval lies within \(O(s)\) of
    one of the \(r-1\) fixed points of \(T_r\).  Thus there are \(O(r)\) bad
    intervals for each \(r\), and \(O(R^2)\) in total.

    On the remaining intervals form a graph by joining distinct \(I,J\) if,
    for some \(2\le r\le R\),
    \[
        I^{(\eta)}\cap T_r^{-1}J^{(\eta)}\neq\varnothing
        \quad\text{or}\quad
        J^{(\eta)}\cap T_r^{-1}I^{(\eta)}\neq\varnothing.
    \]
    For fixed \(I\) and \(r\), the image \(T_r(I^{(\eta)})\) has length
    \(O(rs)\) and meets \(O(r)\) partition intervals.  The inverse image
    \(T_r^{-1}(I^{(\eta)})\) has \(r\) components and also meets \(O(r)\)
    partition intervals.  Hence the graph has maximum degree \(O(R^2)\).
    A greedy independent set has size \(\gg M/R^2\).  By construction, any
    one or two of its intervals satisfy \eqref{eq:family-separation}.
\end{proof}

\begin{proposition}[Uniform avoidance for separated targets]
\label{prop:divisible-avoidance}
    With the notation of \cref{lem:separated-family}, let \(H\) be one member
    of \(\mathcal C_N\), or the union of two distinct members.  Then,
    uniformly in \(H\),
    \begin{equation}\label{eq:divisible-avoidance}
        \Leb_1\{x:a_nx\notin H\text{ for }1\le n\le N\}
        =
        \exp\left(-N\Leb_1(H)+O((\log N)^{-3})\right).
    \end{equation}
\end{proposition}

\begin{proof}
    Put \(p=\Leb_1(H)\), so \(p\in\{s,2s\}\).  For each \(1\le n\le N\),
    let \(\ell_n\) be the least positive integer such that
    \[
        Q_n:=\frac{a_{n+\ell_n}}{a_n}\ge Q.
    \]
    Because every successive quotient is at least~\(2\),
    \[
        \ell_n\le D:=\lceil\log_2Q\rceil=O(\log N).
    \]
    Let \(H_{n,-}\) be the union of the \(Q_n\)-adic intervals contained in
    \(H\), and let \(H_{n,+}\) be the union of those that meet \(H\).
    Each of these sets is a union of at most two circular intervals.  Moreover,
    \begin{equation}\label{eq:grid-sandwich}
        H_{n,-}\subset H\subset H_{n,+}\subset H^{(\eta)},
        \qquad
        \bigl|\Leb_1(H_{n,\pm})-p\bigr|\le\frac4Q.
    \end{equation}

    Define
    \[
        A_n^\pm=\{x:a_nx\in H_{n,\pm}\}.
    \]
    By \eqref{eq:mixed-tail}, the event \(A_n^\pm\) is measurable with respect to
the digit sigma-algebra
\[
    \sigma(\xi_n,\xi_{n+1},\ldots,\xi_{n+\ell_n-1}).
\]
We join \(m\) and \(n\) if the digit windows
\[
    W_m=\{m,\ldots,m+\ell_m-1\},
    \qquad
    W_n=\{n,\ldots,n+\ell_n-1\}
\]
intersect.  Since the digits \((\xi_j)\) are independent, the graph obtained
in this way is a dependency graph in the sense of
\cref{def:dependency-graph}.  Moreover, since \(\ell_n\le D\) for every
\(n\), a fixed window can intersect only windows with indices
\(m\in[n-D+1,n+D-1]\).  Hence the maximum degree is at most \(D_0:=2D\).

    Consider an adjacent pair \(m<n\), and set \(r=a_n/a_m\).  If \(r\le R\),
    then \eqref{eq:family-separation} and \eqref{eq:grid-sandwich} imply
    \[
        \PP(A_m^\pm\cap A_n^\pm)=0.
    \]
    If \(r>R\), multiplication by \(a_m\) preserves Lebesgue measure and
    \(a_nx=r(a_mx)\pmod1\).  Therefore \cref{lem:integer-overlap} gives
    \[
        \PP(A_m^\pm\cap A_n^\pm)
        \le p_*^2+\frac{2p_*}{r},
        \qquad p_*:=p+\frac4Q.
    \]
    For each fixed \(m\), the quotients \(a_n/a_m\) at least double with
    \(n\), whence
    \[
        \sum_{\substack{n>m\\a_n/a_m>R}}\frac{a_m}{a_n}\le\frac2R.
    \]
    Thus the edge-intersection sum in \cref{lem:local-dependence} satisfies
    \begin{equation}\label{eq:Delta-divisible}
        \Delta\ll ND_0p_*^2+\frac{Np_*}{R}.
    \end{equation}

    Apply \cref{lem:local-dependence} separately to the plus and minus
    events, using the degree bound \(D_0\).  Here
    \(\varepsilon=4/Q\), \(D_0=O(\log N)\),
    \(p_*=O((\log N)/N)\), and
    \[
        \frac NQ+ND_0p_*^2+Np_*^2+\frac{Np_*}{R}
        =O((\log N)^{-3}).
    \]
    Hence
    \[
        \PP\left(\bigcap_{n=1}^N(A_n^\pm)^c\right)
        =\exp\left(-Np+O((\log N)^{-3})\right).
    \]
    Finally,
    \[
        \bigcap_n(A_n^+)^c
        \subset
        \{a_nx\notin H\text{ for all }n\le N\}
        \subset
        \bigcap_n(A_n^-)^c,
    \]
    so the same estimate holds for the middle event.
\end{proof}

\subsection{The sharp lower bound}

\begin{proof}[Sharp lower bound in \cref{thm:divisible}]
    Fix \(0<\tau<1\), and use the family \(\mathcal C_N\) from
    \cref{lem:separated-family}.  Let
    \[
        K_N(x)
        =
        \#\{I\in\mathcal C_N:I\cap A_N(x)=\varnothing\}.
    \]
    Since \(Ns=(1+o(1))\tau\log N\),
    \cref{prop:divisible-avoidance} and \eqref{eq:family-size} give
    \begin{align}
        \E K_N
        &\gg
        \frac{M}{R^2}\exp\left(-Ns-O((\log N)^{-3})\right)\nonumber\\
        &\gg
        \frac{N^{1-\tau}}{(\log N)^9}
        \longrightarrow\infty.
        \label{eq:divisible-EK}
    \end{align}
    For distinct \(I,J\in\mathcal C_N\), applying the same proposition to
    \(H=I\cup J\) gives
    \[
        \PP(I\cap A_N(x)=J\cap A_N(x)=\varnothing)
        =\exp\left(-2Ns+O((\log N)^{-3})\right).
    \]
    Consequently,
    \begin{equation}\label{eq:divisible-variance}
        \frac{\Var(K_N)}{(\E K_N)^2}
        \ll
        \frac1{\E K_N}+\frac1{(\log N)^3}.
    \end{equation}

    Let \(N_m=\lfloor\alpha^m\rfloor\), with \(\alpha>1\).  By
    \eqref{eq:divisible-EK} and \eqref{eq:divisible-variance},
    \[
        \sum_{m=1}^\infty
        \PP\{K_{N_m}=0\}<\infty.
    \]
    Borel--Cantelli therefore implies that, for almost every \(x\),
    \(K_{N_m}(x)>0\) for all sufficiently large \(m\).  Thus
    \[
        G_{N_m}(x)\ge s_{N_m}(\tau)
        =(1+o(1))\tau\frac{\log N_m}{N_m}.
    \]
    Interpolating between consecutive terms of the geometric sequence as in
    the proof of \eqref{eq:main-liminf} yields
    \[
        \liminf_{N\to\infty}\frac{NG_N(x)}{\log N}\ge\frac\tau\alpha.
    \]
    Letting \(\alpha\downarrow1\) and \(\tau\uparrow1\) through countable
    sequences gives
    \begin{equation}\label{eq:div-lower}
        \liminf_{N\to\infty}\frac{NG_N(x)}{\log N}\ge1.
    \end{equation}
\end{proof}

Combining \eqref{eq:div-upper} and \eqref{eq:div-lower} proves
\cref{thm:divisible}.

\paragraph*{Acknowledgements.}
This collaboration started when the authors were visiting the Tsinghua Sanya International Mathematics Forum. We thank K. Khanin for asking us about the divisible case $a_n | a_{n+1}$.
The research of Y. Peres was supported by National Natural Science Foundation of China grant RFIS-W2531011.


\bibliographystyle{plain}
\bibliography{ref} 
\end{document}